\documentclass[a4paper,11pt,reqno]{amsart}
\usepackage[centertags]{amsmath}
\usepackage{amsfonts,amssymb,amsthm} 
\usepackage[dvips]{graphicx} 
\usepackage{psfrag}
\usepackage[english]{babel}
\usepackage{newlfont}
\usepackage{mathrsfs}
\usepackage{color}
\usepackage[body={16cm,22.5cm},centering]{geometry} 
\usepackage{fancyhdr}
\newtheorem{theorem}{Theorem}

\newtheorem{lemma}{Lemma}

\newtheorem{definition}{Definition}

\newcommand{\vphi}{\varphi}

\newcommand{\R}{\mathbb{R}}

\newcommand{\Ha}{\mathcal{H}}
\newcommand{\T}{\mathbb{T}}

\newcommand{\la}{\langle}
\newcommand{\ra}{\rangle}
\mathchardef\emptyset="001F

\begin{document}

\title[Remark on a nonlocal isoperimetric problem]{Remark on a nonlocal isoperimetric problem}

\author{Vesa Julin}

\address{Department of Mathematics and Statistics, 
University of Jyv\"{a}skyl\"{a}, Finland}
\email{vesa.julin@jyu.fi}

\maketitle

\begin{center}
\begin{minipage}{12cm}
\small{ 
\noindent {\bf Abstract.} 
We consider  isoperimetric problem with a nonlocal repulsive term given by the Newtonian potential. We prove that 
regular critical sets of the functional are analytic. This optimal regularity  holds also for critical sets of the 
Ohta-Kawasaki functional. We also prove that when the strength of the nonlocal part is small the ball is the only possible stable critical set.

}
\end{minipage}
\end{center}

\bigskip



\section{Introduction}
In this short note we study critical sets of the functional
\begin{equation} \label{energy}
J(E) = P(E) +\gamma \int_{E}\int_{E} G(x,y) dxdy
\end{equation}
where $G(\cdot,\cdot)$ is the standard Newtonian kernel
\[
G(x,y) = \begin{cases}  
 \frac{1}{2\pi} \log \left( \frac{1}{|x-y|} \right) \qquad (n=2) \\
\frac{1}{n(n-2)\omega_n} \frac{1}{|x-y|^{n-2}}\qquad (n\geq 3)
\end{cases}
\]
and $P(E)$ denotes the surface measure, or the perimeter of the set $E$.  This model was first introduced by Gamov \cite{Ga} in the physically relevant case $n=3$ to model the stability of  atomic nuclei. It also rises as a ground state  problem from the Ohta-Kawasaki functional introduced by Ohta and Kawasaki \cite{OK}  to  model diblock copolymers. In the periodic setting the Ohta-Kawasaki  functional can be written as
\begin{equation} \label{energy2}
J_{\T^n}(E) = P(E) +\gamma \int_{\T^n}\int_{\T^n} G_{\T^n}(x,y) u_E(x)u_E(y) dxdy,
\end{equation}
where $G_{\T^n}(\cdot,\cdot)$ is the Green's function in the flat torus and $u_E = 2\chi_E-1$. Both with \eqref{energy}  and \eqref{energy2} we are interested in minimizing the  functional under 
 volume constraint.

There has been an increasing interest  among mathematicians to study the above  functionals \cite{AFM, ACO, BC, CP, CS, CiSp, FL, Go, GMS1, GMS2, KnMu1, KnMu2, LO, StTo, Topa}.  Besides from the obvious physical applications, the main motivation to study   \eqref{energy}  and \eqref{energy2}  is that they  feature the  competition between a short range  interfacial  force, described here  by the perimeter,  which prefers the minimizer to be smooth and connected and a long range repulsive force which prefers the minimizer to be scattered. Indeed, under volume constraint   the ball minimizes the perimeter   by the isoperimetric inequality, while it  maximizes  the nonlocal part  (see e.g. \cite{LL}).

By a scaling argument we notice that when the volume is small the nonlocal term in \eqref{energy} becomes small. This suggest that for $n \geq 3$  the ball should be  the minimizer of \eqref{energy}  under volume constraint when the volume  is small, or equivalently when $\gamma$ is small. This was first proved in \cite{vesku, KnMu2}   and  generalized  in  \cite{BC} to  more general potentials and in \cite{FFMMM} to  nonlocal perimeter. On the other hand when the volume is large the repulsive term becomes stronger and it was proved in  \cite{KnMu2, LO} for $n=3$  that the minimization problem does not have a solution. 

In this note we are interested in critical sets which are not necessarily minimizers. The first result of this paper concerns the regularity of critical sets. To state the result we denote the Newtonian potential by  
\begin{equation} \label{potential}
v_E(x) =  \int_{E} G(x,y) dy.
\end{equation}
The Euler equation associated with  \eqref{energy}  can be then be  written as 
\begin{equation}
\label{euler}
H_E + 2 \gamma v_E = \lambda \qquad \text{on } \, \partial E,
\end{equation}
where $H_E$ is the mean curvature. We say that a  $C^2$-regular set (the boundary is a $C^2$-hypersurface)  is critical if it satisfies \eqref{euler}.  It was proved in \cite{AFM, CiSp, KnMu2} that regular critical sets are $C^{3,\alpha}$-regular for every $\alpha \in (0,1)$ and  then   in \cite{JP} that they are in fact $C^{\infty}$-regular.   We use the method developed in \cite{KLM} to  prove the sharp regularity of   critical sets.
\begin{theorem} \label{regularity}
If  $E \subset \R^n$ is  a regular critical set of \eqref{energy} then it is  analytic. 
\end{theorem}
We note that the above result holds also for critical sets of the Ohta-Kawasaki functional \eqref{energy2}. For the minimizers of \eqref{energy} and \eqref{energy2} we obtain that they are analytic up to a singular set which Hausdorff dimension is at most  $n-8$. Theorem \ref{regularity} can also be applied to 
 improve  the partial regularity result for general non-smooth critical sets in  \cite{GoVo}.

Our second result concerns the uniqueness of stable critical sets.  The study of critical sets is   mathematically interesting for two reasons.  First, it is closely related to the stability of the Alexandroff theorem on sets of constant mean curvatures, since in the case  $\gamma=0$ we know by Alexandroff theorem that the only connected critical  sets  are balls. Second, if we do not have any constraint on $\gamma$ in \eqref{energy} then the family of  possible critical sets is much richer than the family of minimizers. Indeed, it easy to see that for large enough $\gamma$ an annulus is a critical set (see \cite{Go}). More interesting examples of critical sets which are  diffeomorphic to the torus are  constructed in \cite{RW2011}. 

In the planar case the functional \eqref{energy} does not have a minimizer due to the logarithmic behavior of the potential. However, it is showed in \cite{Go} that in the plane when $\gamma$ is small the disk is the unique critical set of  \eqref{energy}.   We would like to have a similar result in higher dimension, but this turns out to be  challenging due to the fact that in higher dimensions there is no stability of the Alexandroff theorem. In \cite{Bu, BuMa}  it is  constructed a sequence of smooth sets with uniformly bounded perimeters which mean curvatures converges uniformly to a constant, but the sets do not converge to a ball. This shows that there cannot be stability of the Alexandroff theorem without further assumptions on the sets (see \cite{CV}). The question of uniqueness of critical sets in higher dimension seems therefore to be  a rather delicate issue. 

Here we prove a weaker result and show that when $\gamma$ is small  the ball is the only possible \emph{stable} critical set. A critical set is stable if the second variation is positive semidefinite (Definition \ref{stable}).   
\begin{theorem} \label{uniqueness}
Let $n\geq 3$ and $L>0$. There exists  $\gamma_0 = \gamma_0(n,L) >0$ such that  if $E \subset \R^n$ with $|E| = |B_1|$ is  a smooth stable critical set of  \eqref{energy} with  $\gamma \leq \gamma_0$ and  $P(E)\leq L$, then $E$ is a ball.
\end{theorem}
The proof of Theorem \ref{uniqueness} is based on the result in \cite{CM}, where it was showed that if a smooth set with uniformly bounded perimeter has almost constant  mean curvature then it it close to a  union of  disjoint tangential balls. We use this result and the stability assumption  to conclude that  
 when $\gamma$ is small  $E$ is actually close to a single ball. The argument is different in the case $n=3$ than in $n\geq 4$. In $n=3$ we use the stability similarly as in  \cite{BdC, Wen} to conclude that $E$ is almost umbilical and therefore it has to be close to a single ball  by  standard estimates for the Willmore functional.   
In the other case  we use an argument similar to  \cite{StZ}  to conclude that  $\partial E$ satisfies a Poincar\'e type inequality which in the case $n\geq 4$ is strong enough to imply 
 that  $E$ is  close to a single ball. Once we have showed that $E$ is close to a single ball, the result follows from the fact that the second variation of the ball is strictly positive and therefore there cannot   be  other critical sets nearby.

\section{Preliminaries}

We first  recall the definition of sets of finite perimeter.  For an introduction to the topic we refer to \cite{Ma}. A  measurable set $E$ has \emph{finite perimeter} if 
\[
P(E) = \sup \Big\{ \int_E \text{div} \varphi\, dx \,\, : \,\, \varphi \in C_0^\infty(\R^n; \R^n), \, \sup|\varphi|\leq 1 \Big\} < \infty.
\]
The quantity $P(E)$ is the perimeter of $E$. If $E$ is a set of finite perimeter then its \emph{reduced boundary} is denoted by $\partial^* E$ and  its perimeter in  $U \subset \R^n$ is
\[
P(E; A) := \Ha^{n-1}(\partial^* E \cap A).
\]
For a Lipschitz set the reduced boundary  agrees with the topological boundary. 

For a sufficiently regular set $E$ we denote by $\nu_E$ the exterior normal of $E$. When no confusion arises we write simply  $\nu$. Given a vector field $X \in C^1(\partial E;\R^n)$ we may extend it to $\R^n$.   We denote  its tangential part on $\partial E$ as $X_{\tau} := X - (X\cdot \nu_E) \nu_E$. It is clear that $X_{\tau}$ does not depend on the extension of $X$. In particular,  we denote by $D_{\tau}$ the tangential gradient operator on $\partial E$ given by $D_{\tau}\vphi := (D\vphi)_\tau$.  Similarly $\text{div}_{\tau}$ and $\Delta_\tau$ denote the tangential divergence and  Laplace-Beltrami operator  respectively and they are defined as $\text{div}_{\tau} X := \text{div}X -  (DX  \nu)\cdot  \nu$ and $\Delta_\tau \vphi := \text{div}_{\tau}(D_\tau \vphi)$. Finally  the \emph{mean curvature} of $E$ is the sum of the principle curvatures which can also be written as $H_E:= \text{div}_{\tau}(\nu_E)$.

Next we define the first and the second variation of \eqref{energy} for a sufficiently regular set $E$. We say that a one-parameter family of diffeomorphisms $\Phi_t :\R^n \to \R^n$, with $t \in (-1,1)$, is admissible flow if 
$\Phi_0(x) = x$ and $|\Phi_t(E)|= |E|$ for all $t$. Let $X$ be the vector field associated with $\Phi_t$, i.e., 
\[
\frac{\partial}{\partial t} \Phi_t = X(\Phi_t).
\] 
Note that the vector field associated with an admissible flow satisfies
\begin{equation} \label{zero average}
\frac{\partial}{\partial t} |\Phi_t(E)| \big|_{t=0} = \int_{\partial E} X \cdot \nu \, d\Ha^{n-1} = 0.
\end{equation}
The first variation of the functional \eqref{energy} is
\[
\frac{\partial}{\partial t} J(\Phi_t(E)) \big|_{t=0} = \int_{\partial E} \text{div}_{\tau} X \,   d \Ha^{n-1}  + 2 \gamma  \int_{\partial E} v_E (X \cdot \nu_E) \,   d \Ha^{n-1},
\]
where the potenial $v_E$ is defined in \eqref{potential}. If $E$ is of class $C^2$ then we may write the first variation as 
\[
\frac{\partial}{\partial t} J(\Phi_t(E)) \big|_{t=0} = \int_{\partial E} (H_E + 2 \gamma   v_E) \,  X \cdot \nu_E \,   d \Ha^{n-1}.
\]
Therefore recalling \eqref{zero average} we define the critical sets of \eqref{energy} as follows.
\begin{definition}
\label{critical}
Let   $E$ be a $C^2$-regular set.  We say that $E$ is a regular critical set if it satisfies the Euler equation \eqref{euler} for some $\lambda \in \R$. The Lagrange multiplier $\lambda$ 
is due to the volume constraint. 
\end{definition}

The second variation of the functional at a general, not necessarily critical, set $E$ is (see \cite{BC})
\[
\begin{split}
\frac{\partial^2}{\partial t^2} J(\Phi_t(E)) \big|_{t=0} = &\int_{\partial E} |D_\tau (X \cdot \nu_E)|^2 - |B_E| (X \cdot \nu_E)^2 + 2 \gamma \partial_\nu v_E (X \cdot \nu_E)^2 \, d\Ha^{n-1} \\
&+ \frac{2\gamma}{n(n-2) \omega_n}  \int_{\partial E} \int_{\partial E}\frac{(X \cdot \nu_E)(x)(X \cdot \nu_E)(y)}{|x-y|^{n-2}} \, d\Ha^{n-1}(x) d\Ha^{n-1}(y)\\
&+\int_{\partial E} (H_E + 2 \gamma v_E) (\text{div} X) (X \cdot \nu_E)d\Ha^{n-1}\\
&- \int_{\partial E} (H_E + 2 \gamma v_E) \text{div}_\tau ( X_\tau (X \cdot \nu_E))d\Ha^{n-1}
\end{split}
\]
If $E$ is a critical set (Definition \ref{critical}) then the two last terms vanish. Indeed this follows from the fact that since the flow is admissible then (see \cite{CS})
\[
0 = \frac{\partial^2}{\partial t^2} |\Phi_t(E)| \big|_{t=0} = \int_{\partial E} (\text{div} X) (X \cdot \nu_E) \, d\Ha^{n-1} = 0.
\]
This leads us to define the  following quadratic form associated with the second variation,
\[
\begin{split}
\partial^2 J(E)[\vphi] := &\int_{\partial E} |D_\tau \vphi|^2 - |B_E| \vphi^2 + 2 \gamma \partial_\nu v_E \vphi^2 \, d\Ha^{n-1} \\
&+ \frac{2\gamma}{n(n-2) \omega_n}  \int_{\partial E} \int_{\partial E}\frac{\vphi(x)\vphi(y)}{|x-y|^{n-2}} \, d\Ha^{n-1}(x) d\Ha^{n-1}(y),
\end{split}
\]
for $\vphi \in H^1(\partial E)$ with $\int_{\partial E} \vphi \, d\Ha^{n-1} = 0$.

\begin{definition}
\label{stable}
Let  $E$ be a regular critical set.  We say that $E$ is a stable if  
\[
\partial^2 J(E)[\vphi] \geq 0
\]
for all $\vphi \in H^1(\partial E)$ with $\int_{\partial E} \vphi \, d\Ha^{n-1} = 0$.
\end{definition}

Finally we need the following simplified version of the Allard's regularity theorem \cite{All}. To this aim for a set of finite perimeter $E \subset \R^n$ and for  $x \in \partial^* E$ and $r>0$ we measure
 the  excess of $E$ at $x$ by
\begin{equation} \label{excess}
\sigma(E,x,r) := \frac{1}{r^{n-1}}\Bigl| P(E, B_r(x)) - \omega_{n-1} r^{n-1} \Bigl|.  
\end{equation}
The following theorem can be found in \cite{All, Sim}
\begin{theorem}[Allard]
\label{allard}
Assume that $E \subset \R^n$ is a set of finite perimeter with bounded mean curvature $|| H_E||_{L^\infty} \leq C$ and let $\alpha  \in (0,1)$.  There exists $\sigma(\alpha)>0$ such that  if  $x \in \partial^* E$ and $r \in (0,1)$ are such that 
\[
\sigma(E,x,r) < \sigma(\alpha)
\]
then $\partial^* E \cap B_{r/2}(x)$ is a $C^{1,\alpha}$-hypersurface. Moreover  up to a rotation  we may write 
\[
\partial^* E \cap  B_{r/2}(x) \subset \{ y' + u(y')e_n \mid y' \in B_{r}^{n-1}(x') \},
\] 
where $x = (x',x_n) \in \R^{n-1} \times \R$, and we have the following estimate
\[
||u||_{C^{1,\alpha}(B_{r/2}^{n-1}(x'))} \leq C \sigma(E,x,r)^{\frac{1}{4(n-1)}}.
\] 
\end{theorem}

\section{Regularity of critical points.}

In this section we prove Theorem \ref{regularity}. As mentioned in the introduction the result follows from the argument developed in \cite{KLM}.

\begin{proof}[\textbf{Proof of Theorem \ref{regularity}}]
Let us fix a point on the boundary $\partial E$ which we may assume to be the origin. By assumption,  $\partial E$ is a $C^2$-hypersurface and therefore it follows from the result in  \cite{JP} that $\partial E$ is  $C^\infty$-regular.  Hence we may assume that there is a radius $r$ and a smooth function
$f: \R^{n-1} \to \R$ such that for $\Gamma = \partial E \cap B_r$ it holds
\[
\Gamma \subset \{ x = (x',x_n) \in \R^{n-1} \times \R \mid x_n = f(x') \}
\]
and that $E$ lies above $\Gamma$.  Moreover, by possible rotating the set, we may assume that $\nabla f(0) = 0$.  

We follow closely the argument from \cite{KLM}. We denote by $B_r^+\subset \R^n$ the upper half ball. Let $\varphi : B^+(0,r) \to \R$ be the solution of 
\[
\begin{cases}
			-\Delta \varphi = 0 , &\text{ in } \,  B_r^+\\
			\varphi(x',0) =  f(x') & \text{ on } \, B_r \cap \{x_n = 0\}.
		\end{cases}
\] 
Since $f$ is smooth then $\varphi$ is smooth up to the boundary. Denote $M :=\sup_{B_r^+} |\nabla \varphi|$. Define  $\Phi^{\pm}: B_r^+ \to \R^n$  by 
\[
\Phi^{+}(x) := (x', \varphi(x) + (M+1)x_n) \qquad \text{and} \qquad \Phi^{-}(x) := (x', \varphi(x) - (M+1)x_n).
\]
Since $\text{det}(D\Phi^{\pm}) = \partial_{x_n}\varphi \pm (M+1) \neq 0$ we may assume, by possibly decreasing $r$, that both $\Phi^+$ and $\Phi^-$ are invertible and denote 
their inverses by $\Psi^+$ and $\Psi^-$. Note also that when $r$ is small it holds $\Phi^+(B_r^+) \subset E$  and $\Phi^-(B_r^+) \subset \R^n \setminus E$.  We define $v_+, v_- : B_r^+ \to \R$ by
\[
v_+(x) :=  v_E(\Phi^+(x)) \qquad \text{and} \qquad v_-(x) :=  v_E(\Phi^-(x)).
\]
Since $\partial E$ is smooth, it follows that the restiction  of $v_E$ to $E$ is smooth up to the boundary  $\partial E$. Similarly
the restriction of $v_E$ to $\R \setminus E$ is smooth up to the boundary. Therefore  we deduce that  $v_+, v_- \in C^\infty(\bar{B}_r^+)$. 

Recall that  $v_E$ is a solution of
\begin{equation} \label{poisson}
- \Delta v_E = \chi_E.
\end{equation}
Let us denote the matrix   $A^\pm(\nabla \varphi) := D\Psi^\pm(\Phi^\pm(x))$ and the vector  $b^\pm(\nabla \varphi) := (\Delta \Psi^\pm)(\Phi^\pm(x))$, where the latter  is a map which coordinate functions
are $\Delta \Psi^\pm_k(\Phi^\pm(x))$.  We conclude from \eqref{poisson} that   $(v_+, v_-, \varphi)$ satisfy the following system of equations in $B_r^+$,
\begin{align*} 
 \text{Trace}\left( A^+(\nabla \varphi)^T D^2 v_+ A^+(\nabla \varphi) \right) + \la b^+(\nabla \varphi),  \nabla v_+ \ra &= 1,  \\
 \text{Trace}\left( A^-(\nabla \varphi)^T D^2 v_- A^-(\nabla \varphi) \right)  + \la b^-(\nabla \varphi), \nabla v_- \ra&=  0,\\
\Delta \varphi &= 0.
\end{align*}
Moreover we deduce from  \eqref{poisson} and from  standard Calderon-Zygmund estimates that  $u \in C^{1,\alpha}(\R^n)$ for every $0 < \alpha< 1$. In particular, this implies that 
\[
v_+ = v_- \qquad \text{and} \qquad  \la  A^+(\nabla \varphi)^T \nabla v_+, \eta(x) \ra = \la  A^-(\nabla \varphi)^T \nabla v_-, \eta(x) \ra 
\] 
on $ B_r \cap \{x_n = 0\}$, where 
\[
\eta(x) =\frac{(-\nabla_{x'}f(x), 1)}{\sqrt{1 + |\nabla_{x'}f|^2}} = \frac{(-\nabla_{x'}\varphi(x), 1)}{\sqrt{1 + |\nabla_{x'}\varphi|^2}}
\]
is the interior normal of $E$. Hence we conclude from the Euler equation \eqref{euler} that   $(v_+, v_-, \varphi)$ satisfy the following boundary conditions on $ B_r \cap \{x_n = 0\}$
\begin{align*} 
\frac{1}{\sqrt{1 + |\nabla_{x'}\varphi|^2}} \left(\Delta_{x'} \varphi - \frac{1}{1 + |\nabla_{x'}\varphi|^2} \la D_{x'}^2 \varphi \nabla_{x'} \varphi, \nabla_{x'} \varphi \ra \right) + 2 \gamma v_+ &= \lambda,\\
A^+(\nabla \varphi)^T \nabla v_+, \eta(x) \ra - \la  A^-(\nabla \varphi)^T \nabla v_-, \eta(x) \ra  &=  0,  \\
v_+ - v_-  &= 0.
\end{align*}

We are now in a position to use classical regularity results (see \cite{M} or \cite[Theorem 2.2]{KLM}) to conclude that $v_+, v_-$ and $\varphi$ are analytic in $\bar{B}_r^+$. Indeed, the functions $v_+, v_- , \varphi$ are smooth up to the boundary $ B_r \cap \{x_n = 0\}$. Moreover the above system of equations  is clearly elliptic and analytic and the boundary conditions  are analytic. Therefore we need only to show that the boundary conditions are complementing at the origin. 

To this aim we recall that we assumed $\nabla_{x'} \varphi(0) = \nabla_{x'} f(0) = 0$. Therefore $\eta(0) =e_n$ and it holds   
\[
 A^+(\nabla \varphi(0)) =  \begin{pmatrix}
 I & 0 \\
0 & \mu_1 
\end{pmatrix} \qquad \text{and} \qquad A^-(\nabla \varphi(0)) =  \begin{pmatrix}
 I & 0 \\
0 & - \mu_2 
\end{pmatrix}
\]
for some  positive numbers  $\mu_1, \mu_2>0$. Here $I$ denotes the $(n-1)$ by $(n-1)$ unit matrix.  Therefore in order to check  that the boundary conditions are complementing at the origin, we need to show that the linearized system
\begin{align*} 
 \sum_{i=1}^{n-1} \partial_{x_ix_i} u_1 + \mu_1^2  \partial_{x_nx_n} u_1 &= 0,  \\
 \sum_{i=1}^{n-1} \partial_{x_ix_i} u_2 + \mu_2^2  \partial_{x_nx_n} u_2 &= 0,\\
\Delta u_3 &= 0
\end{align*}
 in $\R_+^n = \{ x \in \R^n \, : \, x_n > 0\}$, with the following  boundary conditions on $\{x_n = 0\}$,
\begin{align*} 
\Delta_{x'} u_3 &= 0,\\
  \mu_1 \partial_{x_n} u_1 +  \mu_2 \partial_{x_n} u_2 &=  0,  \\
u_1 - u_2  &= 0
\end{align*}
does not have a nontrivial bounded exponential solution of the form
\[
u_j(x) = e^{i \xi \cdot x} \phi_j(x_n)
\]
where $\xi \neq 0$ is orthogonal to $e_n$.

We argue by contradiction and assume that such a solution exists. It follows from the equation for $u_3$   that
\[
\phi_3''(x_n) - |\xi|^2\phi_3(x_n) = 0.
\]
Since we assumed $u_3$, and thus $\phi_3$, to be bounded, we deduce from the above equation that there is a constant $a_3$ such that $\phi_3(t)= a_3 e^{-|\xi|t}$. 
The boundary condition for $u_3$ yields $|\xi|^2\phi_3(0) = 0$, which implies $\phi_3 \equiv 0$.

The equations for $u_1$ and $u_2$ give
\[
\mu_1^2  \phi_1''(x_n) - |\xi|^2 \phi_1(x_n) = 0 \qquad \text{and}\qquad \mu_2^2  \phi_2''(x_n) - |\xi|^2 \phi_2(x_n) = 0.
\]
Since $u_1$ and $u_2$ are bounded, the above equations imply that   there are constants $a_1,a_2$ such that
\[
\phi_1(t) = a_1 e^{-\frac{|\xi|}{\mu_1}t} \qquad \text{and}\qquad \phi_2(t) = a_2 e^{-\frac{|\xi|}{\mu_2}t}.
 \]
The boundary conditions for  $u_1$ and $u_2$ can be written as
\[
\mu_1 \phi_1'(0) =-\mu_2 \phi_2'(0)  \qquad \text{and}\qquad \phi_1(0) = \phi_2(0).
\]
These imply $a_1 = a_2 = 0$, which means that $u_1$ and $u_2$ are zero. Hence the boundary conditions are complementing and the claim follows from result in \cite{M}  (see also \cite[Theorem 2.2]{KLM}).
\end{proof}

\section{Uniqueness of stable critical points.}

Throughout the section we assume  the dimension to be  higher than two, i.e., $n \geq 3$. We begin with two lemmas.  First we show that when $\gamma$ is small then balls are the only  critical points near the unit ball.  This follows from the fact that when $\gamma$ is small  the second variation of the unit ball is strictly positive. 
\begin{lemma} \label{no other critical}
Let $p >n$. There is $\delta_0>0$ such that if $E$ is a regular critical set with $|E| = |B_1|$  such that  
\[
\partial E = \{ \psi(x)x + x \, \mid x \in \partial B_1  \} \qquad \text{and} \qquad ||\psi||_{W^{2,p}}\leq \delta_0,
\]
then $E = B_1(y)$ for some $y \in B_{\delta}(0)$.
\end{lemma}

\begin{proof}
The proof of this lemma can be  essentially found in \cite[Proof of Theorem~3.9]{AFM}  (see also \cite[Theorem  3.11]{BC}). From the results in \cite{AFM, BC} we conclude that  if $E$ is $W^{2,p}$-close to $B_1$ and $|E| = |B_1|$, then  we may find a point $z \in B_{\delta_0}(0)$  and a  vector field $X$ such that the associated flow $\Phi$, 
\[
\frac{\partial }{\partial t} \Phi_t = X(\Phi_t), \qquad \Phi_0(x)= x
\]
satisfies $|\Phi_t(B_1)| = |B_1|$ for every $t  \in [0,1]$,  $\Phi_1(B_1)= E + z$, and 
\[
\frac{d^2 }{dt^2}J(\Phi_t(B_1))\geq c |B_1 \Delta (E + z)|^2
\]
for all $t\in (0,1)$. Assume that $E$ is a regular critical set, which is not a ball. By the criticality of $B_1$ we have $\frac{d }{dt}J(\Phi_t(B_1)) |_{t=0} = 0$. Therefore since $E$ is not a translate of the unit ball
the above inequality gives $\frac{d }{dt}J(\Phi_t(B_1)) |_{t=1} >0$. This implies that $E + z$, and in turn $E$, is not a critical set, which is a contradiction. 
\end{proof}

We state another lemma in which we evaluate the nonlocal  terms in the second variation formula. 
\begin{lemma} \label{estimate second}
Let $E$ be as in the statement of Theorem \ref{uniqueness} and let $\gamma \leq 1$.  Assume  $\vphi \in H^1(\partial E) \cap L^\infty(\partial E)$ has zero average $\int_{\partial E} \vphi \, d\Ha^{n-1}= 0$. There exists a constant  $C$ which depends only  on the dimension and on $L$ such that 
\[
\int_{\partial E} |D_\tau \vphi|^2 - |B_E| \vphi^2  \, d\Ha^{n-1} \geq - C \gamma  ||\vphi||_{L^\infty}^2.
\]
\end{lemma}

\begin{proof}
Since $E$ is critical there exists $\lambda$ such that 
\[
H_E + 2\gamma v_E = \lambda \qquad \text{on } \, \partial E.
\] 
Let us first estimate  the Largange multiplier $\lambda$ and prove that there exists a constant $C$, which depends on $n$, such that 
\begin{equation} \label{bound lagrange}
\Bigl| \lambda - \frac{n-1}{n}\frac{P(E)}{|E|} \Bigl| \leq C \gamma.
 \end{equation}
We write the functional \eqref{energy} as $J(E) = P(E) + \gamma \text{NL}(E)$, where 
\[
\text{NL}(E) := \frac{1}{n(n-2)\omega_n}   \int_{E}\int_{E} \frac{dxdy}{|x-y|^{n-2}}. 
\]
For $t \in \R$ small  we denote $E_t := (1+t)E$. By  scaling of the functional we get
\[
J(E_t) = (1+t)^{n-1}P(E) + (1+t)^{n+2} \gamma \text{NL}(E). 
\]
Therefore 
\begin{equation} \label{easy scaling}
\frac{d}{dt} \bigl|_{t=0}J(E_t) = (n-1)P(E) + (n+2) \gamma  \text{NL}(E).
\end{equation}
On the other hand if we  choose vector field $X= x$ and  denote  its associated flow by $\Phi$,
\[
\frac{\partial }{\partial t} \Phi_t = X(\Phi_t), \qquad \Phi_0(x)= x,
\]
then we get by the first variation formula and by the Euler equation  that 
\[
\begin{split}
 \frac{d}{dt} \bigl|_{t=0}J(\Phi_t(E)) = \int_{\partial E}  (H_E + 2\gamma v_E) \la x,\nu \ra \, \Ha^{n-1}  = \lambda  \int_{\partial E}  \la x,\nu \ra \, \Ha^{n-1}  = n \lambda |E|.
\end{split}
\]
However, it is easy to see that $\Phi_t(E) = (1+t)E + o(t)$ and therefore  $\frac{d}{dt} \bigl|_{t=0}J(E_t)  = \frac{d}{dt} \bigl|_{t=0}J(\Phi_t(E))$. Thus  we obtain by the previous equality and by \eqref{easy scaling} that 
\[
 n \lambda |E| =  (n-1)P(E) + (n+2) \gamma  \text{NL}(E).
\]
Recall that  the ball maximizes the nonlocal part of the functional, i.e., $\text{NL}(E) \leq \text{NL}(B_1)$ for every $|E| = |B_1|$.
Hence we get  \eqref{bound lagrange}.

Let $\vphi \in H^1(\partial E)$ be as in the statement of the lemma. The stability of $E$ (Definition \ref{stable}) yields
\[
\int_{\partial E} |D_\tau \vphi|^2 - |B_E| \vphi^2 + 2\gamma \partial_\nu v_E \vphi^2 \, d\Ha^{n-1} +  \frac{2\gamma}{n(n-2) \omega_n}  \int_{\partial E} \int_{\partial E} \frac{\vphi(x)\vphi(y)}{|x-y|^{n-2}} \, d\Ha^{n-1}(x) d\Ha^{n-1}(y) \geq 0.
\] 
Since $v_E$ is a solution of 
\[
-\Delta v_E = \chi_E \qquad \text{in }\, \R^n
\] 
it follows from standard rearrangement result  \cite{Talenti} that for every $|E|=|B_1|$ it holds  $||v_{E}||_{L^\infty}\leq ||v_{B_1}||_{L^\infty}\leq C$. Moreover, by differentiating \eqref{potential} and arguing as in \cite[Proposition 2.1]{BC} we get 
\begin{equation} \label{calderon}
 ||v_{E}||_{L^\infty} + ||\nabla v_{E}||_{L^\infty}\leq C.
\end{equation}
Therefore the claim follows once we show that for every $x \in \partial E$ we have
\begin{equation} \label{to show}
 \int_{\partial E}\frac{d\Ha^{n-1}(y)}{|x-y|^{n-2}} \leq  C P(E).
\end{equation}

Let us fix $x \in \partial E$. It follows from the Euler equation, from \eqref{bound lagrange} and \eqref{calderon} that $|H_E| \leq C_0$.  Therefore we have the following monotonicity formula \cite{Dl}
\begin{equation} \label{monotoni}
s \mapsto \frac{P(E, B(x,s))}{s^{n-1}} e^{C_0s} \quad \text{ is nondecreasing on } \, s \in (0,\infty).
\end{equation}
Let us define $r_k = 2^{-k}$ for $k= 0,1,\dots$. It follows from \eqref{monotoni} that  for every $k$ we may  estimate
\[
 \int_{\partial E \cap B_{r_k}(x) \setminus B_{r_{k+1}}(x)} \frac{ d\Ha^{n-1}(y) }{|x-y|^{n-2}} \leq 2^{n-2} \frac{P(E, B_{r_k}(x))}{r_k^{n-2}} \leq C P(E, B_1(x)) r_k.
\]
Therefore we get
\[
\begin{split}
 \int_{\partial E \cap B_1(x)}\frac{d\Ha^{n-1}(y) }{|x-y|^{n-2}} &\leq \sum_{k=0}^\infty   \int_{\partial E \cap B_{r_k}(x) \setminus B_{r_{k+1}}(x) } \frac{d\Ha^{n-1}(y) }{|x-y|^{n-2}} \\
&\leq C P(E, B_1(x)) \sum_{k=0}^\infty    r_k =  C P(E, B_1(x)).
\end{split}
\]
Hence the estimate \eqref{to show} follows since trivially 
\[
 \int_{\partial E \setminus B_1(x)}\frac{d\Ha^{n-1}(y)}{|x-y|^{n-2}} \leq   P(E).
\] 
\end{proof}

We are now ready to prove Theorem \ref{uniqueness}.

\begin{proof}[\textbf{Proof of Theorem \ref{uniqueness}}]
We argue by contradiction and assume that there is a sequence $\gamma_k \to 0$ and  associated  smooth stable critical sets $E_k$  with $P(E_k) \leq L$ such that none of them is a ball. 
 The idea is to use the result from \cite{CM} to conclude that $E_k$ convergences  to a set $E$ which is a union of balls which are tangentially connected. We then use the stability assumption to conclude that $E$ is in fact a single ball. By Allard's regularity theorem and the Euler equation \eqref{euler} we then get that  $E_k \to B_1$  in $W^{2,p}$ for every $p>n$, which   contradicts Lemma~\ref{no other critical}.

First, by criticality 
\begin{equation} \label{euler2}
H_{E_k} + 2\gamma_k v_{E_k} = \lambda_k \qquad \text{on } \, \partial E_k.
\end{equation}
Therefore by \eqref{bound lagrange} and \eqref{calderon} we deduce 
\begin{equation} \label{lagrange estimate}
\frac{1}{C} \leq H_{E_k} \leq C.
\end{equation}
By  Topping's inequality \cite{Top}  we have  
\begin{equation} \label{topping}
\text{diam}(E_k) \leq \int_{\partial E_k} H_{E_k}^{n-2} \, d \Ha^{n-1} \leq CP(E_k) \leq CL
\end{equation}
and thus the sets are uniformly bounded.

Let us   show that $\partial E_k$ is connected when $k$ is large. We do this by using the idea from \cite{StZ}. We argue by contradiction and assume that there are at least two components  $\Gamma_1, \Gamma_2$ of  $\partial E_k$. It follows from the curvature bounds  \eqref{lagrange estimate} and from an  isoperimetric type estimate in  \cite{Alm} that 
$\Ha^{n-1}(\Gamma_i) \geq P(B_r)>0$ for $i=1,2$, where the radius of $B_r$ is chosen such that $\frac{n-1}{r}=C$, where  $C$ is the constant in \eqref{lagrange estimate}.  We choose locally constant testfunction in the second variation formula such that $\varphi = 1$ on $ \Gamma_1$ and $\varphi = -\alpha$ on $ \Gamma_2$, where $\alpha >0$ is chosen such that $\int_{\partial E_k} \varphi = 0$.  Note that by the previous discussion we have a uniform bound on  $\alpha$ from above and below. Therefore  Lemma \ref{estimate second} yields
\begin{equation} \label{second taas}
- \int_{\Gamma_1} |B_{E_k}|^2 \, d \Ha^{n-1} \geq C \gamma_k.
\end{equation} 
On the other hand, by \eqref{lagrange estimate} we have that 
\[
|B_{E_k}|^2 \geq \frac{H_{E_k}^2}{n-1} = \frac{1}{(n-1)C^2}, 
\]
which contradicts \eqref{second taas} when $k$ is large. Thus we conclude that the boundary of $E_k$ is connected.

We may  assume, by extracting a subsequence, that there exists a set $E \subset \R^n$ of finite perimeter such that $E_k \to E$ in $L^1$. 
It follows from \eqref{calderon} and \eqref{euler2}  that the sets $E_k$ have almost constant mean curvature and therefore it follows from  \cite[Theorem 1.1]{CM} that $E$ is a union of disjoint balls with equal radii, i.e., there exists $N \in \mathbb{N}$  and a family of disjoint  balls $B(x_1,r), \dots, B(x_N,r)$ such that 
\[
E= \bigcup_{i=1}^N B(x_i,r).
\] 
Each ball in the family is tangent to another one, by which we mean that for every $B(x_i,r)$ there exists $B(x_j,r)$, $j \neq i$, such that  $\bar{B}(x_i,r) \cap \bar{B}(x_j,r) \neq \emptyset$.
 Let us denote by $\Sigma$ the set of all tangent  points in $\bar{E}$, i.e.,
\[
\Sigma := \bigcup_{\underset{ i\neq j}{i,j=1} }^N \bar{B}(x_i,r) \cap \bar{B}(x_j,r). 
\]
Moreover it follows from   \cite[Theorem 1.1]{CM} and from the diameter bound \eqref{topping}  that 
\begin{equation} \label{CM estimates}
\lim_{k \to \infty}P(E_k) = P(E) = N P(B_r) \qquad \text{and}\qquad  \lim_{k \to \infty} \text{hd}(\partial E_k, \partial E) = 0,
\end{equation}
where 'hd' denotes the Hausdorff distance between two sets.  Note that by \eqref{lagrange estimate} the radii of the balls are bounded from below $r \geq 1/C$ and their number  is bounded $N\leq C$.  Let $0 < \delta << r$ be a small number which we choose later. Let us denote the $\delta$-neighborhood of  $\Sigma$ by $\mathcal{N}_{\delta}(\Sigma)$, i.e., 
\begin{equation} \label{n delta}
\mathcal{N}_{\delta}(\Sigma) :=  \bigcup_{y \in \Sigma} B(y,\delta).
\end{equation}
Similarly  let us denote the $\delta$-neighborhood of $\partial E$ by $\mathcal{N}_{\delta}(\partial E)$ and finally
\begin{equation} \label{u delta}
U_\delta :=  \mathcal{N}_{\delta}(\partial E) \setminus \mathcal{N}_{\delta}(\Sigma).
\end{equation}
From \cite[Theorem 1.1]{CM} we  have that  $\partial E_k \cap U_\delta$ is a $C^{1,\alpha}$-manifold when $k$ is large, and it  converges to  $\partial E \cap U_\delta$ in $C^{1,\alpha}$-sense. By this we mean that there is a sequence of  $C^{1,\alpha}$-diffeomorphisms $\Phi_k: \partial E \cap U_\delta  \to  \Phi_k(\partial E \cap U_\delta) \subset \partial E_k$ such that
\[
\lim_{k \to \infty}||\Phi_k - \text{Id}||_{C^{1,\alpha}}=0.
\]

Let us next show that the Lagrange multipliers converges $\lambda_k \to \lambda$  and  that $\lambda$ is the Lagrange multiplier of the limit set $E$. Indeed, let us fix a vector field $X \in C^\infty(\R^n;\R^n)$ such that $X=0$ in  $\mathcal{N}_{\delta}(\Sigma)$. Then by the Euler equation \eqref{euler2} and by the previous $C^{1,\alpha}$-convergence we get
\[
\begin{split}
\int_{\partial E_k} (\lambda_k - 2\gamma_k v_{E_k}) (X \cdot \nu_{E_k}) \, d \Ha^{n-1} &= \int_{\partial E_k} H_{E_k} (X \cdot \nu_{E_k}) \, d \Ha^{n-1} = \int_{\partial E_k} \text{div}_{\tau_{E_k}} X  \, d \Ha^{n-1} \\
&\to \int_{\partial E} \text{div}_{\tau_E} X  \, d \Ha^{n-1} \qquad \text{as }\, k \to \infty \\
&=\int_{\partial E} \lambda  (X \cdot \nu_{E}) \, d \Ha^{n-1}. 
\end{split}
\]
Since this holds for any vector field $X$ and since $\gamma_k v_{E_k} \to 0$ we conclude that $\lambda_k \to \lambda$.  Arguing as in \cite[Lemma 7.2]{AFM} this in turn implies that $\partial E_k \cap U_\delta$ converges to $\partial E \cap U_\delta$ in $W^{2,p}$ for every $p >n$, i.e., 
\begin{equation} \label{w 2 p}
\lim_{k \to \infty}||\Phi_k - \text{Id}||_{W^{2,p}(\partial E \cap U_\delta)}=0 \qquad \text{for every } \, p > n, 
\end{equation}
where $\Phi_k: \partial E \cap U_\delta  \to  \Phi_k(\partial E \cap U_\delta) \subset \partial E_k$.
 
Let us next prove that the limit set $E$ consists only on a single ball. This will follow from the stability of $E_k$.   We argue by contradiction and assume that $N \geq 2$ for  $E = \bigcup_{i=1}^N B(x_i,r)$.  
The argument is different in the case  $n=3$ and  $n\geq 4$. 

\bigskip

\noindent \textbf{The case $n=3$}. In this case we use an  argument similar to \cite{BdC, Wen} to show that $E_k$ are nearly umbilical. This together with standard estimates on Willmore energy imply that the limit set $E$ has to be a single ball.

Let us choose a testfunction $\vphi = x\cdot \nu - \sigma$ in the second variation condition. Here $\sigma \in \R$ is chosen such that $\int_{\partial E_k} \vphi\, d \Ha^2 = 0$. By the divergence theorem we may solve $\sigma$ 
\[
\sigma P(E_k) = \int_{\partial  E_k} \sigma \, d \Ha^2 =   \int_{E_k}  x\cdot \nu \, dx =  \int_{E_k} \text{div}(x) \, dx = 3 |E_k|.
\] 
Hence 
\begin{equation}
 \label{sigma thing}
\sigma =\frac{3 |E_k|}{P(E_k) } \geq c >0.
\end{equation}
Note that by \eqref{topping} we may assume, by translating the sets $E_k$, that  $||\vphi||_{L^\infty} \leq C$ on $\partial E_k$.

Let us fix a basis $\{e_1,e_2,e_3 \}$ in $\R^3$. We denote $\nu_j = \nu \cdot e_j$ and $\delta_j f := \nabla_\tau f \cdot e_j$ for any given smooth function $f$ on $\partial E_k$. We may extend $f$ smoothly to a neighborhood of $\partial E_k$ and then  $\delta_j f :=  \partial_{x_j}f -   (\nabla f \cdot e_j) \,  \nu_j$.    Next we use a well know geometric equality \cite[Eq.~(10.16)]{Giu} and the Euler equation \eqref{euler2}  to deduce 
\[
\Delta_\tau \nu_j = -|B_{E_k}|^2\nu_j + \delta_j H_{E_k} =  -|B_{E_k}|^2\nu_j  - 2\gamma_k \delta_j  v_{E_k} 
\]
for  $j = 1, 2, 3$. It is well known, and straighforward to see, that  $\Delta_\tau x_j = - H_{E_k}  \nu_j$ on $\partial E_k$. Therefore we have 
\begin{equation}
 \label{geometric}
\begin{split}
\Delta_\tau \vphi &= \sum_{j=1}^3 (\Delta_\tau \nu_j x_j +2 \nabla_\tau \nu_j \cdot \nabla_\tau x_j + \nu_j \Delta_\tau x_j)    \\
&=  \sum_{j=1}^3 (\Delta_\tau \nu_j x_j +2 \nabla_\tau \nu_j \cdot e_j  - H_{E_k} \nu_j^2)   \\
&=  \sum_{j=1}^3 \Delta_\tau \nu_j x_j + 2H_{E_k} - H_{E_k}  \\
&=  - |B_{E_k}|^2\vphi - |B_{E_k}|^2 \sigma +  H_{E_k} - 2\gamma_k \sum_{j=1}^3 \delta_j  v_{E_k} x_j.
\end{split}
\end{equation}
We multiply \eqref{geometric} by $\vphi$, integrate over $\partial E_k$ and get
\[
\begin{split}
\int_{\partial E_k} |D_\tau \vphi|^2 - |B_{E_k}|^2 \vphi^2 \, d \Ha^{2} &= \int_{\partial E_k} \sigma |B_{E_k}|^2 \vphi  -  (H_{E_k}  - 2\gamma_k \sum_{j=1}^3 \delta_j  v_{E_k} x_j) \vphi \, d \Ha^{2}  \\
&= \int_{\partial E_k} \sigma |B_{E_k}|^2 \vphi  -  (\lambda_k - 2\gamma_k v_{E_k} - 2\gamma_k \sum_{j=1}^3 \delta_j  v_{E_k} x_j) \vphi \, d \Ha^{2} \\
&\leq  \sigma  \int_{\partial E_k} |B_{E_k}|^2  \vphi \, d \Ha^{2} + C \gamma_k,
\end{split}
\]
where the last inequality follows from \eqref{calderon} and from the fact that $\vphi$ has zero average. Therefore Lemma \ref{estimate second} and \eqref{sigma thing}  imply
\begin{equation}
 \label{from_stability}
-\int_{\partial E_k} |B_{E_k}|^2  \vphi \, \Ha^{2} \leq C\gamma_k.
\end{equation}
Note that \eqref{bound lagrange}, the Euler equation \eqref{euler2}, and \eqref{sigma thing} yield
\[
\frac{1}{\sigma} \leq \frac{H_{E_k}}{2} + C \gamma_k.
\]
We integrate \eqref{geometric} over $\partial E_k$ and obtain by the above inequality that
\[
\begin{split}
-\int_{\partial E_k} |B_{E_k}|^2  \vphi \, d \Ha^{2} &\geq \int_{\partial E_k}  \sigma  |B_{E_k}|^2 - H_{E_k} \, d \Ha^{2} - C\gamma_k  \\
&\geq \sigma  \int_{\partial E_k}  |B_{E_k}|^2  - \frac{H_{E_k}^2}{2}\,  d \Ha^{2}  - C\gamma_k\\
&= \frac{\sigma}{2}  \int_{\partial E_k}   (\kappa_1 - \kappa_2)^2 \, d \Ha^{2} - C\gamma_k,
\end{split}
\]
where $\kappa_1$ and $\kappa_2$ are the principle curvatures. Hence we have by \eqref{from_stability} that
\begin{equation}
 \label{umbilical}
 \int_{\partial E_k}  (\kappa_1 - \kappa_2)^2 \,  d \Ha^{2}  \leq C\gamma_k.
\end{equation}
When $\gamma_k$ is small  it follows  from \eqref{umbilical} and  e.g. from \cite[Lemma 2.2]{DlM} that $\partial E_k$ has genus zero.  Therefore \eqref{umbilical} and Gauss-Bonnet theorem yield
\begin{equation}
 \label{gauss}
 \int_{\partial E_k} |B_{E_k}|^2 \, d \Ha^2 \leq  8 \pi + C\gamma_k.
\end{equation}

On the other hand  from the $W^{2,p}$-convergence  \eqref{w 2 p} it follows that  
\[
\begin{split}
\lim_{k \to \infty} \int_{\partial E_k \cap U_\delta} |B_{E_k}|^2 \, d \Ha^2 = \int_{\partial E \cap U_\delta} |B_{E}|^2 \, d \Ha^2  =  \frac{2}{r^2}  P(E; U_\delta).
\end{split}
\]
Recall the  definition of $ \mathcal{N}_{\delta}(\Sigma)$ and $U_\delta$, \eqref{n delta} and \eqref{u delta}. From the monotonicity formula \eqref{monotoni} we 
get that $P(E; \mathcal{N}_{\delta}(\Sigma)) \leq C \delta^2$. Therefore we have  
\[
P(E; U_\delta) \geq P(E) - P(E; \mathcal{N}_{\delta}(\Sigma)) \geq P(E) - C\delta^2 = 4 N\pi r^2 -C\delta^2.
\] 
Hence we get 
\[
\lim_{k \to \infty} \int_{\partial E_k \cap U_\delta} |B_{E_k}|^2 \, d \Ha^2 \geq 8 N \, \pi  - C\delta^2,
\]
which contradicts \eqref{gauss} if $N \geq 2$ when $\delta$ is small and we conclude that $E$ is a ball. By translation we may assume that $E = B_1$. 

Let us next fix $\alpha \in (0,1)$ and show that for every $k$ large there exists a function $\psi_k : \partial B_1 \to \partial E_k$ such that  
\begin{equation}
 \label{see yks alfa}
\partial E_k = \{ \psi_k(x)x + x \, \mid x \in \partial B_1  \} \qquad \text{and} \qquad ||\psi||_{C^{1,\alpha}}\leq c,
\end{equation}
where $c$ is idependent of $k$. To this aim we fix $x \in \partial B_1$ and let $\sigma(\alpha)$ be the constant from Theorem \ref{allard}. Choosing  $\rho>0$ small we get   
$\sigma(B_1,x,\rho) <\frac{\sigma(\alpha)}{2}$, where the excess $\sigma$ is defined in \eqref{excess}. The claim \eqref{see yks alfa} then follows from Theorem \ref{allard} once we show that 
\begin{equation}
 \label{lokaali peri}
\lim_{k \to \infty} P(E_k, B_\rho(x)) =P(B_1, B_\rho(x)).
\end{equation}
By the lower semicontinuity of the perimeter we have that $\lim_{k \to \infty} P(E_k, B_\rho(x)) \geq P(B_1, B_\rho(x))$. Suppose that  $\lim_{k \to \infty} P(E_k, B_\rho(x)) > P(B_1, B_\rho(x))$. Then again by the lower semicontinuity we have   $\lim_{k \to \infty} P(E_k, \R^n \setminus B_\rho(x)) \geq  P(B_1,  \R^n \setminus B_\rho(x))$. This implies that 
\[
\lim_{k \to \infty} P(E_k)  > P(B_1)
\]   
which is a contradiction since by \eqref{CM estimates} we have $\lim_{k \to \infty} P(E_k)  = P(B_1)$. Hence we have \eqref{lokaali peri}, and in turn \eqref{see yks alfa}.

We need yet to show that  $\psi_k \to 0$ in $W^{2,p}$ for every $p>n$, where $\psi_k$ are the functions from \eqref{see yks alfa}.  By the stability of the isoperimetric inequality proved in \cite{FJ} and by  \eqref{see yks alfa} we have 
\[
P(E_k) - P(B_1) \geq c \min_{y \in \R^n}\int_{\partial E_k}\Big|\nu_E(x)- \frac{x-y}{|x-y|}\Big|^2\,d\Ha^{n-1}(x) \geq  c ||D \psi_k||_{L^2(\partial B_1)}^2.
\]
Therefore by  \eqref{CM estimates} and  \eqref{see yks alfa} we conclude that  $\psi_k \to 0$ in $C^1(\partial B_1)$, up to a subsequence.   Arguing as in \eqref{w 2 p} we finally conclude that $\psi_k \to 0$ in $W^{2,p}$ for every $p>n$.
By Lemma \ref{no other critical} we have $E_k = B_1(y_k)$ for some $y_k$ when $k$ is large, which is a contradiction and the result follows in the case $n=3$.

\bigskip

\noindent \textbf{The case $n \geq 4$}. In this case we apply an argument similar to  \cite{StZ} where we use the stability condition to conclude that $\partial E_k$ satisfy a Poincar\'e type inequality. From this we may conclude that  the limit set $E$ is connected and it has to be a single ball. 

First we assume that the limit set $E$ is just a union of two tangent balls  since the general case $N \geq 3$ follows from a similar argument. Indeed, as we already mentioned, we will show that the stability of $E_k$ imply that Poincar\'e type inequality holds for $\partial E$ and obtain a contradiction by analyzing  $E$ near the  set $\Sigma$ which is a finite union of isolated points.  Since the argument is  local near the tangent  points  we may assume that $N =2$.  By translating and rotating we may further assume that 
\[
E = B(x_1,r) \cup B(x_2,r)
\]
for  $x_1= re_n$ and $x_2= -re_n$, where $e_n$ is the $x_n$-coordinate direction. Note that in this case $\Sigma = \{0\}$.

Let us define a function $f: \R \to [-1,1]$ as
\[
f(t) := \begin{cases}  
\min \{ \frac{r}{\delta^2}t -1, 1\} \, &\text{when }\, t \geq\frac{\delta^2}{r} \\
 0  \, &\text{when }\, - \frac{\delta^2}{r} < t < \frac{\delta^2}{r} \\
-\min \{ \frac{r}{\delta^2}t -1, 1\} \, &\text{when }\, t \leq -\frac{\delta^2}{r}. 
\end{cases}
\]
 Let us further define $\vphi : \R^n \to \R$ as  $\vphi(x) := f(x_n)$. Note  that $\vphi \equiv 0$ on $\partial E \cap B_{\delta}$ and  that $\int_{\partial E} \vphi \, d \Ha^{n-1}= 0$. Let us show 
that the following inequality holds on $\partial E$
\begin{equation}
 \label{poincare type}
 \int_{\partial E \setminus B_\delta } |D_\tau \vphi|^2  - \frac{2}{r^2} \vphi^2  \, d \Ha^{n-1} \geq 0.
\end{equation}

To this aim denote $\bar{f}_k = \int_{\partial E_k} f(x_n) \, d \Ha^{n-1}$ and  choose $\vphi_k(x) = f(x_n) - \bar{f}_k$ in the second variation formula of $E_k$. Lemma \ref{estimate second} gives
\[
\int_{\partial E_k} |D_{\tau_k} \vphi_k|^2 - |B_{E_k}|^2 \vphi_k^2  \, d\Ha^{n-1} \geq - C \gamma_k.
\] 
Since $\vphi \equiv 0$ on $\partial E \cap B_{\delta}$ we deduce from  the Hausdorff convergence in \eqref{CM estimates} that  $D_{\tau_k} \vphi_k = 0$ on $\partial E_k \cap  B_\delta$ when $k$ is large. Hence we  have that 
\begin{equation}
 \label{poincare type k}
\int_{\partial E_k \setminus B_\delta} |D_{\tau_k} \vphi_k|^2 - |B_{E_k}|^2 \vphi_k^2  \, d\Ha^{n-1} \geq - C \gamma_k.
\end{equation}
By the $W^{2,p}$-convergence  of $\partial E_k \setminus B_\delta$ proved  in \eqref{w 2 p}  it  follows that 
\[
\lim_{k \to \infty}\int_{\partial E_k \setminus B_\delta} |D_{\tau_k} \vphi_k|^2 \, d\Ha^{n-1} =  \int_{\partial E \setminus B_\delta } |D_\tau \vphi|^2   \, d \Ha^{n-1}
\]
and
\[
\lim_{k \to \infty}\int_{\partial E_k \setminus B_\delta}|B_{E_k}|^2 \vphi_k^2  \, d\Ha^{n-1} =  \int_{\partial E \setminus B_\delta } |B_{E}|^2 \vphi^2   \, d \Ha^{n-1}.
\]
Hence we may pass to the limit  in \eqref{poincare type k} and get \eqref{poincare type}.

Let  us study the inequality \eqref{poincare type} and  define $\Sigma_{3\delta} := \partial E \cap B_{3\delta}$. Note that $\partial E \setminus B_{3\delta}$ is disconnected and $\vphi \equiv 1$ on $ \partial E \setminus B_{3\delta} \cap \{x_n >0\}$ and  $\vphi \equiv -1$ on $ \partial E \setminus \Sigma_{3\delta} \cap \{x_n <0\}$.  In particular,   $|D_\tau \vphi| = 0$ on $\partial E \setminus B_{3\delta}$.  It is also clear that on $\Sigma_{3\delta}$ for every tangent vector $\tau_i$  it holds  $|\tau_i \cdot e_n| \leq C \delta$.  Therefore we may estimate the tangential gradient of $\vphi$ on $\Sigma_{3\delta}$  as
\[
|D_\tau \vphi| \leq \sum_{i=1}^{n-1} |f'| |\tau_i \cdot e_n| \leq C \delta^{-1}.
\] 
On the other hand we may estimate the area  of $\Sigma_{3\delta}$ by
\[
\Ha^{n-1}(\Sigma_{3\delta}) =  P(E; B_{3\delta}) \leq C \delta^{n-1}.
\]
Therefore \eqref{poincare type} gives 
\[
0 \leq  C\Ha^{n-1}(\Sigma_{3\delta})   \delta^{-2}  - \frac{2}{r^2}  \Ha^{n-1}( \partial E \setminus B_{3\delta}) \leq C  \delta^{n-3}  - \frac{2}{r^2}  P(B_1)
\]
when $\delta >0$ is small.  This is clearly a contradiction in the case $n \geq 4$ when $\delta>0$ is  small enough. Hence $E$ is a single ball and 
we conclude, as in the case $n=3$, that $E_k \to B_1$ in $W^{2,p}$ for every $p>n$. Again as in the case $n=3$, we conclude the proof by Lemma \ref{no other critical}.
\end{proof}

\section*{Acknowledgment}
 This  work was supported by the  Academy of Finland grant 268393.


\begin{thebibliography}{100}

\bibitem{AFM}
\textsc{E. Acerbi , N. Fusco \& M. Morini}, 
\emph{Minimality via second variation for a nonlocal isoperimetric problem.} Comm. Math. Phys. \textbf{322} (2013), 515--557.

\bibitem{ACO}
\textsc{G. Alberti, R. Choksi \& F. Otto}, 
\emph{Uniform energy distribution for an isoperimetric problem with long-range interactions}. 
J. Amer. Math. Soc. \textbf{22} (2009),  569--605. 

\bibitem{All}
\textsc{W.K. Allard}, 
\emph{On the first variation of a varifold.} Ann. Math. \textbf{95} (1972), 417--491.

\bibitem{Alm}
\textsc{F. Almgren}, 
\emph{Optimal isoperimetric inequalities.} Indiana Univ. Math. J. \textbf{35} (1986),  451--547. 


\bibitem{BdC}
\textsc{J.L. Barbosa \& M do Carmo},
\emph{Stability of hypersurfaces with constant mean curvature.} 
Math. Z.  \textbf{185}  (1984),  339--353. 


\bibitem{BC}
\textsc{M. Bonacini \& R. Cristoferi,}
\emph{Local and global minimality results for a nonlocal isoperimetric problem on $\R^N$}. SIAM J. Math. Anal.  \textbf{46}  (2014), 2310--2349.

\bibitem{Bu}
\textsc{A. Butscher},
\emph{A gluing construction for prescribed mean curvature}.
Pacific J. Math. \textbf{249} (2011), 257--269. 


\bibitem{BuMa}
\textsc{A. Butscher \& R. Mazzeo},
\emph{CMC hypersurfaces condensing to geodesic segments and rays in Riemannian manifolds}. 
Ann. Sc. Norm. Super. Pisa Cl. Sci. (5) \textbf{11} (2012),  653--706. 

\bibitem{CP}
\textsc{R. Choksi and M. Peletier}, 
\emph{Small Volume Fraction Limit of the Diblock Copolymer Problem: I Sharp Inteface Functional} SIAM J. Math Analysis \textbf{42-3} (2010), 1334--1370 

\bibitem{CS}
\textsc{R. Choksi \& P. Sternberg},
\emph{On the first and second variations of a nonlocal isoperimetric problem}.  J. Reine Angew. Math. \textbf{611} (2007), 75--108.


\bibitem{CiSp}
\textsc{M. Cicalese \& E. Spadaro},
\emph{Droplet minimizers of an isoperimetric problem with long-range interactions.}  Comm. Pure Appl. Math. \textbf{66}   (2013),  1298--1333.


\bibitem{CM}
\textsc{G. Ciraolo \& F. Maggi,}
\emph{On the shape of compact hypersurfaces with almost constant mean curvature.}  Preprint 2015.

\bibitem{CV}
\textsc{G. Ciraolo \& L. Vezzoni,}
\emph{A sharp quantitative version of Alexandrov's theorem via the method of moving planes}. Preprint 2015. 

\bibitem{Dl}
\textsc{C. De Lellis,} 
\emph{Rectifiable sets, densities and tangent measures.}  Zurich Lectures in Advanced Mathematics. European Mathematical Society (EMS), Z\"urich,  (2008). 


\bibitem{DlM}
\textsc{C. De Lellis \& S.  M\"uller,} 
\emph{Optimal rigidity estimates for nearly umbilical surfaces.} J. Differential Geom.  \textbf{69}  (2005),  75--110. 


\bibitem{FFMMM}
\textsc{A. Figalli, N. Fusco, F. Maggi, V. Millot \& M. Morini}, 
\emph{Isoperimetry and stability properties of balls with respect to nonlocal energies}, Comm. Math. Phys. \textbf{336} (2015), 441--507.


\bibitem{FL}
\textsc{R.L. Frank \& E.H. Lieb},
\emph{A compactness lemma and its application to the existence of minimizers for the liquid drop model}, Preprint 015).

\bibitem{FJ}
\textsc{N. Fusco \& V. Julin},
\emph{A strong form of the Quantitative Isoperimetric inequality,} Calc. Var. PDEs. \textbf{50}  (2014), 925--937.



\bibitem{Ga}
 \textsc{G. Gamow},
\emph{Mass defect curve and nuclear constitution.} 
Proceedings of the Royal Society of London. Series A, \textbf{126} (1930), 632--644.

\bibitem{Go}
 \textsc{D. Goldman},
\emph{Uniqueness results for critical points of a non-local isoperimetric problem via curve shortening.} Preprint  2012.


\bibitem{GMS1}
 \textsc{D. Goldman, C.B Muratov \& S. Serfaty},  
\emph{The $\Gamma$-limit of the two-dimensional Ohta-Kawasaki energy. I. Droplet density.} Arch. Ration. Mech. Anal. \textbf{210} (2013),  581--613.


\bibitem{GMS2}
 \textsc{D. Goldman, C.B Muratov \& S. Serfaty},  
\emph{The $\Gamma$-limit of the two-dimensional Ohta-Kawasaki energy. Droplet arrangement via the renormalized energy}. 
Arch. Ration. Mech. Anal. \textbf{212} (2014),  445--501. 


\bibitem{GoVo}
 \textsc{D. Goldman \& A. Volkmann},
 \emph{A short note on the regularity of critical points to the Ohta-Kawasaki energy}, Preprint  2014.

\bibitem{Giu}
 \textsc{E. Giusti},
 \emph{Minimal surfaces and functions of bounded variation.}  Monographs in Mathematics, 80. Birkh\"auser Verlag, Basel, (1984).

\bibitem{vesku}
 \textsc{V. Julin},
 \emph{Isoperimetric problem with a Coulomb repulsive term}. Indiana Univ. Math. J. \textbf{63} (2014), 77--89.

\bibitem{JP}
\textsc{V. Julin \& G. Pisante,}
 \emph{Minimality via second variation for microphase separation of diblock copolymer melts.} to appear in J. Reine Angew. Math.


\bibitem{KnMu1}
\textsc{H. Kn\"upfer \& C.B. Muratov},
\emph{On an isoperimetric problem with a competing non-local term. I. The planar case.} Comm. Pure Appl. Math.  66  (2013),  1129--1162.

\bibitem{KnMu2}
\textsc{H. Kn\"upfer \& C.B. Muratov},
\emph{On an isoperimetric problem with a competing non-local term. II. The general case. }  Comm. Pure Appl. Math.  \textbf{67}  (2014),  1974--1994. 




\bibitem{KLM}
\textsc{H. Koch, G. Leoni \& M. Morini}, 
\emph{On optimal regularity of free boundary problems and a conjecture of De Giorgi.} Comm. Pure Appl. Math.  \textbf{58} (2005), 1051--1076. 

\bibitem{LL}
\textsc{E.H. Lieb \& M. Loss}, 
\emph{Analysis}. Second edition. Graduate Studies in Mathematics, 14. American Mathematical Society, Providence, RI, (2001).


\bibitem{LO}
\textsc{J. Lu \& F. Otto},
\emph{Nonexistence of minimizer for Thomas-Fermi-Dirac-Von Weizs\"acker model.}  Comm. Pure Appl. Math. \textbf{67} (2014),  1605--1617. 


\bibitem{Ma}
\textsc{F. Maggi},
\emph{Sets of finite perimeter and geometric variational problems. An introduction to geometric measure theory}. 
Cambridge Studies in Advanced Mathematics, 135. Cambridge University Press, Cambridge (2012).

\bibitem{M}
\textsc{C.B. Jr. Morrey}, 
\emph{Multiple integrals in the calculus of variations.} Die Grundlehren der mathematische Wissenschaften, Band 130, Springer-Verlag New York, Inc., New York (1966). 



\bibitem{OK}
\textsc{T. Ohta \& K. Kawasaki},
\emph{Equilibrium morphologies of block copolymer melts}. 
Macromolecules,  \textbf{19} (1986),  2621--2632. 


\bibitem{RW2011}
\textsc{X. Ren \& J. Wei},
\emph{A toroidal tube solution to a problem involving mean curvature and Newtonian potential}. 
Interfaces Free Bound. \textbf{13} (2011),  127--154. 

\bibitem{Sim}
\textsc{L. Simon},
\emph{Lectures on geometric measure theory.} volume 3 of \emph{Proceedeings of the Centre for Mathematical Analysis}. Australian National University, Center for Mathematical Analysis, Canberra (1983).

\bibitem{StTo}
\textsc{P. Sternberg \& I. Topaloglu}, 
\emph{On the global minimizers of a nonlocal isoperimetric problem in two dimensions}. 
Interfaces Free Bound. \textbf{13} (2011),  155--169. 
 
\bibitem{StZ}
\textsc{P. Sternberg \& K.  Zumbrun},
\emph{A Poincar\'e inequality with applications to volume-constrained area-minimizing surfaces}.
 J. Reine Angew. Math.  \textbf{503} (1998), 63--85.


\bibitem{Talenti}
\textsc{G. Talenti},
\emph{Elliptic equations and rearrangements.} Ann. Scuola Norm. Sup. Pisa Cl. Sci. \textbf{3}  (1976),  697--718.

\bibitem{Topa}
\textsc{I. Topaloglu}, 
\emph{On a nonlocal isoperimetric problem on the two-sphere.}
Commun. Pure Appl. Anal. \textbf{12} (2013),  597--620. 


\bibitem{Top}
\textsc{P. Topping}, 
\emph{Relating diameter and mean curvature for submanifolds of Euclidean space.} Comment. Math. Helv.  \textbf{83}  (2008),  539--546.

\bibitem{Wen}
\textsc{H.C. Wente},
\emph{A note on the stability theorem of J. L. Barbosa and M. Do Carmo for closed surfaces of constant mean curvature. }
Pacific J. Math.  \textbf{147}  (1991),   375--379. 

\end{thebibliography}
\end{document}